\documentclass[11pt]{article}
\usepackage{setspace}
\usepackage{amssymb,amsmath,amsthm,tabu,longtable}
\usepackage{subfigure}
\usepackage{calc}
\usepackage{verbatim}
\usepackage{enumerate}
\usepackage{epsfig}
\usepackage{graphicx}
\usepackage{graphics}
\usepackage {tikz}
\usetikzlibrary{automata,arrows,calc,positioning}
\usetikzlibrary{decorations.pathreplacing}
\usepackage{xcolor}\usepackage[hmargin=2cm,vmargin=2.5cm]{geometry}
\usepackage{chngcntr}

\counterwithin{figure}{section}
\newcommand{\tiltfrac}[2]{\left.#1\middle/#2\right.}
\newtheorem{definition}{Definition}[section]
\newtheorem{lemma}[definition]{Lemma}

\newtheorem{theorem}[definition]{Theorem}
\newtheorem{corollary}[definition]{Corollary}
\newtheorem{observation}[definition]{Observation}
\newtheorem{problem}{Problem}
\newtheorem{conjecture}[definition]{Conjecture}

                                {\null\hfill$\Box$\par\medskip\medskip\medskip\medskip}

\newcommand{\avd}{\text{\rm{avd}}}

\newcommand{\A}[3][\mathcal{D}_G]{\mathcal{A}_{#3}(#2, #1)}
\newcommand{\N}[3][\mathcal{D}_G]{\mathcal{N}_{#3}(#2, #1)}
\newcounter{eqcount}

\makeatletter
\renewcommand{\pod}[1]{\mathchoice
  {\allowbreak \if@display \mkern 18mu\else \mkern 8mu\fi (#1)}
  {\allowbreak \if@display \mkern 18mu\else \mkern 8mu\fi (#1)}
  {\mkern4mu(#1)}
  {\mkern4mu(#1)}
}
\tikzset{bignode/.style={minimum size=3em,}}
\begin{document}

\title{A Tight Upper Bound on the Average Order of Dominating Sets of a Graph}
\author{Iain Beaton\\
\small Department of Mathematics \& Statistics\\
\small Acadia University\\
\small Wolfville, NS Canada\\
\small iain.beaton@acadiau.ca\\
\and
Ben Cameron\\
\small Department of Computing Science\\
\small The King's University\\
\small Edmonton, AB Canada\\
\small ben.cameron@kingsu.ca\\
}

\date{\today}

\maketitle

%=================================================================
%=================================================================
\begin{abstract}
In this paper we study the average order of dominating sets in a graph, $\avd(G)$. Like other average graph parameters, the extremal graphs are of interest. Beaton and Brown (2021) %showed that $\overline{K}_n$ maximized the average order of dominating sets among all graphs with $n$ vertices.
conjectured that for all graphs $G$ of order $n$ without isolated vertices, $\avd(G) \leq 2n/3$. Recently, Erey (2021) proved the conjecture for forests without isolated vertices. In this paper we prove the conjecture and classify which graphs have $\avd(G) = 2n/3$. We also use our bounds to prove the average version of Vizing's Conjecture.\\

\end{abstract}

%%%%%%%%%%%%%%%%%%%%%%%%%%%%%%%%%%%%%%%%%%%%%%%%%%%%%%%%%%%%%%%%%%%%%%%%%%%%
\section{Introduction}\label{sec:intro}
%%%%%%%%%%%%%%%%%%%%%%%%%%%%%%%%%%%%%%%%%%%%%%%%%%%%%%%%%%%%%%%%%%%%%%%%%%%%
%In studying the structural properties of graphs, maximal and minimal sizes of these structures in a given graph are often of interest. However, the \textit{average} size of these structures have begun to receive a tremendous amount of interest and can often be more instructive.%maybe cut these first two sentences
The study of average graph parameters is often traced back to the introduction of the Wiener index of a graph by Wiener~\cite{1947Wiener} in 1947 due to its close connection to the average distance of a graph. These parameters have fascinating applications to chemistry~\cite{2002Rouvray}, but even without the chemical applications, average graph parameters have proven important for the study of structural graph theory. Some of the most natural questions to ask about any average graph parameter is which graphs maximize and minimize it and what are its maximum and minimum values among all graphs of the same order. While the former question can (but certainly does not always) have predictable answers, proving them is often notoriously difficult. Doyle and Graver~\cite{1977Doyle} showed for connected graphs of order $n$ the average distance in a graph was maximized by a path, with value $\tiltfrac{(n+1)}{3}$, and minimized by the complete graph with value $1$. In the three quarters of a century since the Wiener index was introduced, many other average graph parameters have been introduced and attracted significant interest in the literature.

In a series of seminal papers \cite{1983Jamison,1984Jamison}, Jamison introduced the mean subtree order of a tree. In these works, it was shown that the mean subtree order of trees is minimized by the path and conjectured that the tree maximizing it must be a caterpillar. The mean subtree order has seen a resurgence of interest in recent years, especially in the so-called Caterpillar Conjecture (see e.g., \cite{2019Mol, 2015Wagner, 2021Cambie, 2010Vince}). Generalizing the mean subtree order from trees to graphs was only recently done by Chin et al.~\cite{2018ChinGordonMacpheeVincent}, where, among other things, a conjecture was posed that would have implied that the path minimized and the complete graph maximized the mean subtree order of connected graphs.  Although this specific conjecture was recently was disproved~\cite{2021Cameron}, it was replaced by a weaker version that would still imply that the complete graph maximizes the mean subtree order of a graph. In another generalization of Jamison's work to all graphs, Kroeker, Mol and Oellermann~\cite{2018Kroeker} introduced the mean connected induced subgraph order of a graph and conjectured that the path minimizes it. While this was very recently proved independently by Haslegrave~\cite{2022Haslegrave_pathminimizes} and Vince~\cite{2022Vince}, the question of which graph maximizes the mean connected induced subgraph order for connected graphs of order $n$ remains open for all $n\ge 10$ ($n\le 9$ is classified in the conclusion of~\cite{2018Kroeker} where the difficulty of this problem is noted). The mean connected induced subgraph order of a graph continues to provoke interest when further restrictions are placed on the graphs~\cite{2020Balodis, 2022Haslegrave_withdegconstraints}. The average size of independent sets was introduced by Davies, Jenssen, Perkins and Roberts~\cite{2017Davies} (see also~\cite{2018Davies}) as the logarithmic derivative of the independence polynomial, motivated by applications to statistical mechanics. Andriantiana, Razanajatovo Misanantenaina and Wagner~\cite{2020Andriantiana_indep} showed that for graphs of a fixed order, the empty graph maximizes and the complete graph minimizes the average size of independent sets. Closely related is the average size of matchings introduced again by Andriantiana, Razanajatovo Misanantenaina and Wagner~\cite{2020Andriantiana_matchings} where it was conversely shown that the empty graph minimizes and the complete graph maximizes.

This paper is focused on the average order of dominating sets of a graph $G$ (denoted $\avd(G)$). The first author and Brown~\cite{2021Beaton} introduced $\avd(G)$ and showed that 

$$\frac{n\cdot 2^{n-1}}{2^n-1}=\avd(K_n)\le \avd(G)\le \avd(\overline{K_n})=n $$

\noindent for all graphs $G$ of order $n$. Additionally, they showed if $G$ had no isolated vertices then $\avd(G) \le 3n/4$ but conjectured this bound could be improved.

\begin{conjecture}[\cite{2021Beaton}]
\label{conj:2n/3}
 Let $G$ be a graph with $n \geq 2$ vertices. If $G$ has no isolated vertices then $\avd(G) \leq \frac{2n}{3}$.
\end{conjecture}

\noindent This improved upper bound was shown to hold for forests without isolated vertices by Erey~\cite{2021Erey} who also gave a structural characterization for the forests achieving the upper bound. %and while the results are specialized to forests, the methods reveal even more about the structure of dominating sets in forests.  
In this paper we prove Conjecture~\ref{conj:2n/3} and completely classify all graphs that achieve $\avd(G)=2n/3$.

This result comes in Section~\ref{sec:avdnorm} and allows for a short proof of an average version of Vizing's Conjecture that we prove in Section~\ref{sec:Vizing}. We conclude by posing some open problems and ideas for future research in Section~\ref{sec:conclusion} including a connection to the unimodality conjecture for domination polynomials. We first include a brief subsection including many of the definitions and notations used, although more substantial technical definitions will be stated as relevant throughout.

\subsection{Definitions and Notation}\label{subsec:defs}
For a graph $G$ containing vertex $v$, $N_G(v) = \{u|uv \in E(G)\}$ denotes the \emph{open neighbourhood} of $v$ while $N_G[v] = N(v) \bigcup \{v\}$ denotes the \emph{closed neighbourhood} of $v$ (we will omit the subscript $G$ when only referring to one graph). The \textit{degree} of a vertex $v$, denoted $\deg(v)$ is defined at $|N(v)|$ and the \textit{minimum degree} of a graph $G$, denoted $\delta(G)$ is $\min_{v\in V(G)}\deg(v)$.  For $S \subseteq V(G)$, the closed neighbourhood $N[S]$ of $S$ is simply the union of the closed neighbourhoods for each vertex in $S$. A \textit{leaf} is a vertex with degree $1$, a \textit{stem} is a vertex adjacent to a leaf, and a \textit{$k$-stem} is a vertex adjacent to exactly $k$ leaves. A subset of vertices $S$ is a \emph{dominating set} of $G$ if $N[S] = V(G)$, that is, every vertex is either in $S$ or adjacent to a vertex in $S$. The \emph{domination number} of $G$, denoted $\gamma(G)$, is the order of the smallest dominating set of $G$.% The study of dominating sets in graphs is extensive (see, for example, \cite{hedetniemi}).
%A dominating with order $\gamma(G)$ is called a \emph{minimum dominating} set. 
Let $\mathcal{D}_G$ denote the collection of dominating sets of $G$. Furthermore let $d_k(G) = |\{S \in \mathcal{D}_G: |S|=k\}|$. Then the \emph{average order of dominating sets} in $G$, denoted $\avd(G)$, is
$$\avd(G)= \frac{\sum\limits_{k=\gamma(G)}^{|V(G)|} kd_k(G)}{\sum\limits_{k=\gamma(G)}^{|V(G)|} d_k(G)},$$
that is, the average cardinality of a dominating set of $G$. For simplicity let 

 $$\Gamma_G= \sum\limits_{k=\gamma(G)}^{|V(G)|} d_k(G)\hspace{2mm} \text{ and } \hspace{2mm} \Gamma'_G=\sum\limits_{k=\gamma(G)}^{|V(G)|} kd_k(G),$$

\noindent therefore we have $\avd(G)=\frac{\Gamma'_G}{\Gamma_G}$.

\section{Proof of Conjecture~\ref{conj:2n/3}}\label{sec:avdnorm}
A crucial notation to the work in~\cite{2021Beaton} and to our work is that of critical vertices in a dominating set. For a fixed dominating set $S$ of a graph $G$, $v\in S$ is \emph{critical} with respect to $S$ if $S-v$ is no longer a dominating set. For a dominating set $S$ of a graph $G$ let $a(S)$ be the set of all critical vertices with respect to $S$, that is,

$$a(S) = \{v \in S : S-v \notin \mathcal{D}_G\}.$$

\noindent The total number of critical vertices amongst all dominating sets in $G$ is a function of $\Gamma'_G$ and $\Gamma_G$.
 
\begin{lemma}[\cite{2021Beaton}]
\label{lem:sumaS}
If graph is a graph $G$ with $n$ vertices, then
$$\sum\limits_{S \in \mathcal{D}_G}|a(S)| = 2\Gamma'_G-n\Gamma_G.$$
\end{lemma}

Let $N(S)$ denote the usual open neighbourhood of the set $S$, $\left(\bigcup_{v\in S}N(v)\right)- S$. We note if $S$ is a dominating set, then we note that $N(S)$ is also the set of vertices not $S$ (i.e. the set $V-S$). Note that for a graph with $n$ vertices

\[\sum\limits_{S \in \mathcal{D}_G}|N(S)| = \sum_{S \in \mathcal{D}_G} (n-|S|) = n\Gamma_G-\Gamma'_G  \refstepcounter{eqcount} \label{eqn:N(S)} \tag{\theeqcount}.\]

\begin{observation}
\label{obs:AvN}
For a graph $G$ with $n$ vertices $ \avd(G) \leq \frac{2n}{3}$ if and only if

$$\sum\limits_{S \in \mathcal{D}_G}|a(S)| \leq \sum\limits_{S \in \mathcal{D}_G}|N(S)|. $$
\end{observation}

\begin{proof}
By Lemma \ref{lem:sumaS} and equation \ref{eqn:N(S)} we have  

$$\sum\limits_{S \in \mathcal{D}_G}|a(S)| = 2\Gamma'_G-n\Gamma_G \text{ }\hspace{2mm}\text{ and }\hspace{2mm}\text{ } \sum\limits_{S \in \mathcal{D}_G}|N(S)| =n\Gamma_G-\Gamma'_G.$$

\noindent Therefore, 

\begin{align*}
& \hspace{-3cm} &  \sum\limits_{S \in \mathcal{D}_G}|a(S)| \leq & \sum\limits_{S \in \mathcal{D}_G}|N(S)|\\
\Leftrightarrow& \hspace{-3cm} & 2\Gamma'_G-n\Gamma_G  \leq & n\Gamma_G-\Gamma'_G  \\
\Leftrightarrow& \hspace{-3cm} & 2\avd(G)-n \leq & n - \avd(G)  \\
\Leftrightarrow& \hspace{-3cm} & \avd(G) \leq & \frac{2n}{3}.  
\end{align*}
\end{proof}
\noindent 

For a dominating set $S$, we will need to consider partitions of both $N(S)$ and $a(S)$ into two sets as was done in~\cite{2021Beaton}. For $N(S)$, we define $N_1(v)$ as the set of vertices outside $S$ which have a single neighbour in $S$, and $N_2(S)$ as the set of vertices outside $S$ which have more than one neighbour in $S$. More precisely,

\vspace{-6mm}

\begin{align*}
N_1(S) &= \{v \in V-S : |N[v] \cap S| = 1 \}\\
N_2(S) &= \{v \in V-S : |N[v] \cap S| \geq 2 \}.
\end{align*}

\noindent Since $S$ is a dominating set, it is clear that $N(S)=N_1(S)\cup N_2(S)$. Similarly, for $a(S)$, we define sets $a_1(S)$ and $a_2(S)$ such that $a(S) = a_1(S) \cup a_2(S)$ as follows:

\vspace{-6mm}

\begin{align*}
a_1(S) &= \{v \in a(S) : N(v) \cap N_1(S) \neq \emptyset \}\\
a_2(S) &= \{v \in a(S) : N(v) \cap N_1(S) = \emptyset \}.
\end{align*} 

%\vspace{-6mm}

Now that we have these partitions, we can state the following lemma that will be useful.

\begin{lemma}[\cite{2021Beaton}]
\label{lem:a1n1}
 Let $G$ be a graph. For any dominating set $S \in D_G$, $|a_1(S)| \leq |N_1(S)|$.
\end{lemma}

Lemma \ref{lem:a1n1} implies that $\sum\limits_{S \in \mathcal{D}_G}|a_1(S)| \leq \sum\limits_{S \in \mathcal{D}_G}|N_1(S)|$ and hence by Observation \ref{obs:AvN} it is sufficient to show $\sum\limits_{S \in \mathcal{D}_G}|a_2(S)| \leq \sum\limits_{S \in \mathcal{D}_G}|N_2(S)|$ to prove that $\avd(G) \leq \frac{2n}{3}$.

The arguments used in \cite{2021Beaton} involved counting the critical vertices in a given dominating set. One of the limitations of this approach was that vertices with low degree do not find themselves in $N_2(S)$ for many dominating sets. More specifically, a vertex of degree one will never be in the set $N_2(S)$ for any dominating set as it does not have two vertices in its open neighbourhood. However, if a vertex with low degree has a neighbour with high degree then they will appear in $a_2(S)$ more frequently than $N_2(S)$. To deal with this we will instead consider ordered pairs $(v, S)$ where $v$ is a vertex in one of $a_1(S)$, $a_2(S)$, $N_1(S)$, or $N_2(S)$. This will allow us to group vertices in such a way that collectively they appear in $N(S)$ more frequently than $a(S)$. For a graph $G$, a subset of its vertices $X \subseteq V(G)$ and a collection of dominating sets $\mathcal{Y} \subseteq \mathcal{D}_G$, let 

\vspace{-6mm}

\begin{align*}
\A[\mathcal{Y}]{X}{1} &= \{(v,S) \in X \times \mathcal{Y} : v \in a_1(S)\}\\
\A[\mathcal{Y}]{X}{2} &= \{(v,S) \in X \times \mathcal{Y} : v \in a_2(S)\}\\
\N[\mathcal{Y}]{X}{1} &= \{(v,S) \in X \times \mathcal{Y} : v \in N_1(S)\}\\
\N[\mathcal{Y}]{X}{2} &= \{(v,S) \in X \times \mathcal{Y} : v \in N_2(S)\}.
\end{align*}

\vspace{-6mm}

\noindent Note that

$$\sum\limits_{S \in \mathcal{Y}}|a_1(S) \cap X| = |\A[\mathcal{Y}]{X}{1}|, \text{ }\hspace{2mm}\text{  }\hspace{2mm}\text{ } \sum\limits_{S \in \mathcal{Y}}|a_2(S) \cap X| = |\A[\mathcal{Y}]{X}{2}|,$$
$$ \sum\limits_{S \in \mathcal{Y}}|N_1(S) \cap X|  = |\N[\mathcal{Y}]{X}{1}| \text{ }\hspace{2mm}\text{ and }\hspace{2mm}\text{ } \sum\limits_{S \in \mathcal{Y}}|N_2(S) \cap X| = |\N[\mathcal{Y}]{X}{2}|.$$

\noindent Additionally, let $\A[\mathcal{Y}]{X}{}=\A[\mathcal{Y}]{X}{1} \cup \A[\mathcal{Y}]{X}{2}$  and $\N[\mathcal{Y}]{X}{}=\N[\mathcal{Y}]{X}{1} \cup \N[\mathcal{Y}]{X}{2}$. The next observation allows us to use results from \cite{2021Beaton} using our refined notation in the case that $\mathcal{Y} =  \mathcal{D}_G$ and $X=V$.

\begin{observation}
\label{obs:NewNotation}
For a graph $G$ then

$$\sum\limits_{S \in \mathcal{D}_G}|a(S)| = |\A{V}{}| \text{ }\hspace{2mm}\text{ and }\hspace{2mm}\text{ } \sum\limits_{S \in \mathcal{D}_G}|N(S)| = |\N{V}{}|$$
\end{observation}

One benefit of this approach is that each of the four sets $\A[\mathcal{Y}]{X}{1}$, $\A[\mathcal{Y}]{X}{2}$, $\N[\mathcal{Y}]{X}{1}$ and $\N[\mathcal{Y}]{X}{2}$ are additive across disjoint sets. For example, if $X_1, X_2 \subseteq V$ and $\mathcal{Y}_1, \mathcal{Y}_2 \subseteq \mathcal{D}_G$ are such that $X_1 \cap X_2 = \emptyset$ and $\mathcal{Y}_1 \cap \mathcal{Y}_2 = \emptyset$ then 

$$|\A[\mathcal{Y}_1 \cup \mathcal{Y}_2]{X_1 \cup X_2}{1}| = |\A[\mathcal{Y}_1]{X_1}{1}| + |\A[\mathcal{Y}_1]{X_2}{1}|  + |\A[\mathcal{Y}_2]{X_1}{1}|  + |\A[\mathcal{Y}_2]{X_2}{1}|.$$

\noindent If $X$ is the singleton set $X=\{v\}$ we will simply write $\A[\mathcal{Y}]{v}{1}$, $\A[\mathcal{Y}]{v}{2}$, $\N[\mathcal{Y}]{v}{1}$, and $\N[\mathcal{Y}]{v}{2}$. Thus, Lemma~\ref{lem:a1n1} can be restated as showing $|\A{V}{1}| \leq |\N{V}{1}|$. Observation~\ref{obs:AvN} then implies is that if $|\A{V}{2}| \leq |\N{V}{2}|$, then $\avd(G) \leq \frac{2n}{3}$. In the next lemma we will give singleton sets $\{v\}$ such that $|\A{v}{2}| \leq |\N{v}{2}|$. 

\begin{lemma}
\label{lem:deg2}
Let $G$ be a graph with vertex $v$ such that $\deg(v) \geq 2$. Then 

$$(2^{\deg(v)}-\deg(v)-1) \cdot |\A{v}{2}| < |\N{v}{2}|.$$

\end{lemma}

\begin{proof}
Let $S$ be a dominating set such that $(v,S) \in \A{v}{2}$. Therefore $v \in a_2(S)$ and $S-v$ is not a dominating set. We begin by showing $S-v$ dominates every vertex in $G$ except $v$. Since $S$ is a dominating set of $G$ then the only vertices which are not dominated by $S-v$ are in $N[v]$. If a vertex in $N[v]$ is not dominated by $S-v$, it must not be in $S-v$. By the definition of $a_2(S)$ we have that $N(v)\cap N_1(S)= \emptyset$ and hence every vertex in $N(v)$ which is not in $S-v$ is in $N_2(S)$. By the definition of $N_2(S)$, every vertex in $N(v)$ which is not in $S$ has another neighbour in $S$ and therefore are still dominated by $S-v$. Therefore the only vertex which is not dominated by $S-v$ is $v$ itself. Thus, for any subset $T \subseteq N(v)$ we have that $(S-v) \cup T$ is a dominating set of $G$. Moreover if $|T| \geq 2$ then $v \in N_2((S-v) \cup T)$ and hence $(v, (S-v) \cup T) \in \N{v}{2}$. 

We now define a mapping that will allow us to find for each dominating set $S$ and $v\in a_2(S)$, a new dominating set without $v$ and with at least two neighbours of $v$. Let $\mathcal{P}_2( N(v))$ denote the collection of all subsets $T \subseteq N(v)$ with $|T|\geq 2$.  Let $f:\A{v}{2} \times \mathcal{P}_2( N(v)) \rightarrow \N{v}{2}$ be the mapping $f((v, S), T)=(v, (S-v) \cup T)$. For any $(v, S_1), (v, S_{2}) \in \A{v}{2}$ and $T_1, T_2 \in \mathcal{P}_2( N(v))$ we have that $(S_1-v)\cap T_1=\emptyset$ and $(S_2-v)\cap T_2=\emptyset$. Therefore, since $|T_1|,|T_2|\ge 2$, $(S_1-v)\cup T_1=(S_2-v) \cup T_2$ only if $S_1=S_2$ and $T_1=T_2$. Thus, $f$ is an injective mapping from $\A{v}{2} \times \mathcal{P}_2( N(v))$ to $\N{v}{2}$. Finally as $|\mathcal{P}_2( N(v))|=2^{\deg(v)}-\deg(v)-1$ we obtain

$$(2^{\deg(v)}-\deg(v)-1) \cdot |\A{v}{2}| \leq |\N{v}{2}|.$$

To show the inequality is strict, we will show that $f$ is not surjective. first note that as $\deg(v) \geq 2$ then $(v,V-v) \in \N{v}{2}$ and hence $|\N{v}{2}|>1$. When $\deg(v) > 2$ then $2^{\deg(v)}-\deg(v)-1) >1$ and hence the inequality is strict. So suppose $\deg(v) = 2$. It suffices to show there is a dominating set $S$ such that $(v,S) \in \N{v}{2}$ but $(v,S)$ is not in the image of $f$. Let $S=N(v) \cup (V-N[N[v]])$. This is a dominating set in $G$ as $N(v)$ dominates every vertex in $N[N[v]]$. Moreover $v \in N_2(S)$ so $(v,S) \in \N{v}{2}$. If $(v,S)$ were in the image of $f$ then $S-N[v]$ would dominate every vertex in $G$ except for $v$. However $S-N[v]=V-N[N[v]]$ and thus does not dominate any vertex in $N(v)$. Therefore $(v,S)$ is not in the image of $f$ and the inequality is strict.

\end{proof}

If $G$ is a graph with $n$ vertices and minimum degree $\delta \geq 2$, then Lemma \ref{lem:deg2} and Lemma \ref{lem:a1n1} imply that $|\A{V}{}| < |\N{V}{}|$. Therefore by Observation \ref{obs:AvN} we have that $\avd(G) < \frac{2n}{3}$. If $G$ has a vertex $v$ of degree one then $|\N{v}{2}|=0$ as no dominating set could possibly omit $v$ while still containing two neighbours of $v$. However, it is still possible for $|\A{v}{2}| > 0$ when $\deg(v)=1$ and hence $|\A{v}{2}| > |\N{v}{2}|$. For example, when $G = K_{1,n-1}$ and $n\geq 3$ then the dominating set $S$ containing only degree one vertices is minimal and thus $a(S)=S$. Moreover, $N_1(S) = \emptyset$ so $a_2(S)=S$ and every degree one vertex $\ell$ of $K_{1,n-1}$ has $|\A{v}{2}| > 0$ but $|\N{v}{2}|=0$. To deal with this, we will group the degree one vertices in $G$ with other vertices in $G$ to form a set $X$ which satisfies $|\A{X}{2}| \leq |\N{X}{2}|$.

For a graph $G$, we call a vertex with degree one a \emph{leaf}. The lone neighbour of a leaf is called a \emph{stem}. For a graph $G$, a \emph{$k$-stem} is stem which is adjacent to $k$ leaves. For a stem $s$, let $L(s)$ denote the collection of leaves adjacent to $s$ and let $L[s]=L(s) \cup \{s\}$. We will now partition $\mathcal{D}_G$ based on which stems $s$ have $L[s] \not\subseteq S$. Suppose $G$ has stems $s_1, s_2, \ldots , s_{\omega}$. For some $I \subseteq \{s_1, s_2, \ldots, s_{\omega}\}$ let $\mathcal{X}_{I}$ denote the collection of dominating sets $S \in \mathcal{D}_G$ with $L[s] \not\subseteq  S$ for all $s \in I$ and $L[s] \subseteq  S$ for all $s \notin I$. That is

$$\mathcal{X}_{I}= \left( \bigcap_{s \in I} \{S \in \mathcal{D}_G : L[s] \not\subseteq  S \} \right) \cap \left( \bigcap_{s \notin I} \{S \in \mathcal{D}_G : L[s] \subseteq  S\} \right).$$

\noindent Since any leaf can only be dominated by itself or its stem, it follows that the collection of all $\mathcal{X}_{I}$ partitions $\mathcal{D}_G$.

\begin{lemma}
\label{lem:kstem}
Let $G$ have stems $s_1, s_2, \ldots , s_{\omega}$. For any $I \subseteq \{s_1, s_2, \ldots, s_{\omega}\}$ and \emph{$k$-stem} $s \in I$ we have

$$|\A[\mathcal{X}_{I}]{L[s]}{}|\leq |\N[\mathcal{X}_{I}]{L[s]}{}|, $$

\noindent and strict inequality $|\A[\mathcal{X}_{I}]{L[s]}{}| < |\N[\mathcal{X}_{I}]{L[s]}{}|$ if $k \geq 3$.
\end{lemma}

\begin{proof}
Note that as $s \in I$ then for any $S \in \mathcal{X}_{I}$ we have that $L[s] \not\subseteq S$. For simplicity let $L=L(s)$. We begin by partitioning $\mathcal{X}_{I}=\mathcal{X}_{s} \cup \mathcal{X}_{L}$ where $\mathcal{X}_{s}=\{S\in \mathcal{X}_{I}: s\not\in S\}$ and $\mathcal{X}_{L}=\{S\in \mathcal{X}_{I}: L\not\subseteq S\}$. Note if $S \in \mathcal{X}_{s}$, then every leaf in $L$ must be in $S$ otherwise it would not be dominated. Moreover, if $S \in \mathcal{X}_{L}$ then $s \in S$ otherwise the leaves of $L$ which are not in $S$ would not be dominated. Thus, $\mathcal{X}_{s}$ and $\mathcal{X}_{L}$ do indeed partition $\mathcal{X}_{I}$. These two cases are illustrated in Figure \ref{fig:L12} where the shaded vertices are the vertices in $S$.

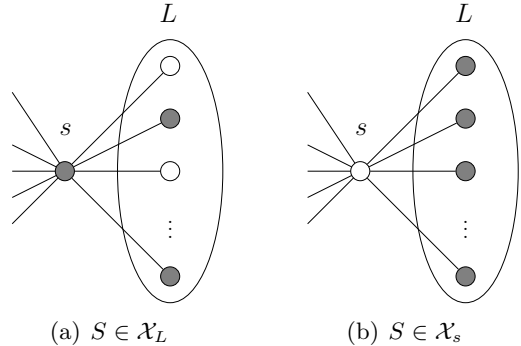
\begin{figure}[h]
\def\c{0.7}
\centering
\subfigure[$S \in \mathcal{X}_{L}$]{
\scalebox{\c}{
\begin{tikzpicture}

\draw[] (2,0) ellipse (1 and 2.5);
\node[text width=1cm] at (2.3,3) {\Large $L$};
\node[text width=1cm] at (0.4,0.8) {\Large $s$};

\path [-] (0,0) edge node {} (-1, 1.5);
\path [-] (0,0) edge node {} (-1, 0.5);
\path [-] (0,0) edge node {} (-1, 0);
\path [-] (0,0) edge node {} (-1, -0.5);
\path [-] (0,0) edge node {} (-1, -1);

\node[shape=circle,draw=black,fill=gray] (1) at (0,0) {};
\node[shape=circle,draw=black,fill=white] (2) at (2,2) {};
\node[shape=circle,draw=black,fill=gray] (3) at (2,1) {};
\node[shape=circle,draw=black,fill=white] (4) at (2,0) {};
\node[shape=circle,draw=black,fill=gray] (5) at (2,-2) {};

\path (4) -- node[auto=false]{\vdots} (5);

\begin{scope}
    \path [-] (1) edge node {} (2);
    \path [-] (1) edge node {} (3);
    \path [-] (1) edge node {} (4);
    \path [-] (1) edge node {} (5);
\end{scope}

\end{tikzpicture}}}
\qquad
\subfigure[$S \in \mathcal{X}_{s}$]{
\scalebox{\c}{
\begin{tikzpicture}
\draw[] (2,0) ellipse (1 and 2.5);
\node[text width=1cm] at (2.3,3) {\Large $L$};
\node[text width=1cm] at (0.4,0.8) {\Large $s$};

\path [-] (0,0) edge node {} (-1, 1.5);
\path [-] (0,0) edge node {} (-1, 0.5);
\path [-] (0,0) edge node {} (-1, 0);
\path [-] (0,0) edge node {} (-1, -0.5);
\path [-] (0,0) edge node {} (-1, -1);

\node[shape=circle,draw=black,fill=white] (1) at (0,0) {};
\node[shape=circle,draw=black,fill=gray] (2) at (2,2) {};
\node[shape=circle,draw=black,fill=gray] (3) at (2,1) {};
\node[shape=circle,draw=black,fill=gray] (4) at (2,0) {};
\node[shape=circle,draw=black,fill=gray] (5) at (2,-2) {};

\path (4) -- node[auto=false]{\vdots} (5);

\begin{scope}
    \path [-] (1) edge node {} (2);
    \path [-] (1) edge node {} (3);
    \path [-] (1) edge node {} (4);
    \path [-] (1) edge node {} (5);
\end{scope}

\end{tikzpicture}}}
\caption{The two cases for when $S \in \mathcal{X}_{L[s]}$}%
\label{fig:L12}%
\end{figure}

We will now show that $(2^{k}-1)|\mathcal{X}_{s}| \leq |\mathcal{X}_{L}|$. Let $S \in \mathcal{X}_{s}$ and note that $(S-L) \cup \{s\}$ is a dominating set of $G$. Moreover for every proper subset $T \subset L$ we have that $(S-L) \cup \{s\} \cup T \in \mathcal{X}_{L}$. Let $\mathcal{P}(L)$ denote the power set of $L$. Now define $f: (\mathcal{P}(L)-L) \times \mathcal{X}_{s} \rightarrow \mathcal{X}_{L}$ as the mapping $f(T, S) = (S-L) \cup \{s\} \cup T$. For any $S_1, S_2 \in \mathcal{X}_{s}$ we have $s\not\in S_1\cup S_2$, $L\subseteq S_1$, and $L\subseteq S_2$, so for any $T_1, T_2 \in \mathcal{P}(L)-L$ we have $(S_1-L)\cup \{s\} \cup T_1=(S_2-L)\cup \{s\} \cup T_2$ only if $S_1=S_2$ and $T_1=T_2$. Thus, $f$ is an injective mapping from $(\mathcal{P}(L)-L) \times \mathcal{X}_{s}$ to $\mathcal{X}_{L}$ and therefore $(2^{k}-1)|\mathcal{X}_{s}| \leq |\mathcal{X}_{L}|$. Additionally let $\mathcal{X}'_{L}$ be the image of $f$ over $(\mathcal{P}(L)-L) \times \mathcal{X}_{s}$. Note that $\mathcal{X}'_{L} \subseteq \mathcal{X}_{L}$ and $|\mathcal{X}'_{L}|=(2^{k}-1)|\mathcal{X}_{s}|$.

We will now consider the cases of dominating sets $S \in \mathcal{X}_{L}- \mathcal{X}'_{L}$, $S \in \mathcal{X}'_{L}$ and $S \in \mathcal{X}_{s}$ and count $a(S) \cap L[s]$ and $N(S)\cap L[s]$ for each one.

For any $S \in \mathcal{X}_{L}- \mathcal{X}'_{L}$ we still have that $S \in \mathcal{X}_{L}$. Therefore $s \in a_1(S)$ and each leaf of $L$ not in $S$ is in $N_1(S)$. Note that the leaves of $L$ in $S$ are not in either $a(S)$ nor $N(S)$. Thus, for every $S \in \mathcal{X}_{L}- \mathcal{X}'_{L}$ we have $|a(S) \cap L[s]|=1$ and $|N(S) \cap L[s]|=|L-S|$. Recall that $|L-S|\geq 1$ as there must be at least one leaf of $L$ not in $S$ by definition. Therefore for $\mathcal{X}_{L}-\mathcal{X}'_{L}$ we have

\[|\A[\mathcal{X}_{L}-\mathcal{X}'_{L}]{L[s]}{}| \leq |\N[\mathcal{X}_{L}-\mathcal{X}'_{L}]{L[s]}{}|. \refstepcounter{eqcount} \label{eqn:leftovers} \tag{\theeqcount}\]

Now consider the dominating sets in $\mathcal{X}'_{L}$. For each dominating set $S \in \mathcal{X}_{s}$, the mapping $f$ assures there are $2^{k}-1$ dominating sets in $\mathcal{X}'_{L}$ of the form $S_T = (S-L)\cup \{s\} \cup T$ each corresponding to a proper subset $T \subset L$. For each $S \in \mathcal{X}_{s}$ and proper subset $T \subset L$ we have $|a(S_T) \cap L[s]|=1$ and $|N(S_T) \cap L[s]|=|L|-|T|=k-|T|$. Therefore,
%$$|\A[\mathcal{X}'_{L}]{L[s]}{}|=(2^{k}-1)|\mathcal{X}_{s}|  \text{ }\hspace{2mm}\text{ and }\hspace{2mm}\text{ } |\N[\mathcal{X}'_{L}]{L[s]}{}|=\sum_{i=0}^{k-1}(k-i) \cdot {k \choose i} \cdot |\mathcal{X}_{s}|=k \cdot 2^{k-1} |\mathcal{X}_{s}|.$$

$$|\A[\mathcal{X}'_{L}]{L[s]}{}|=(2^{k}-1)|\mathcal{X}_{s}|$$

\noindent and

\begin{align*}
|\N[\mathcal{X}'_{L}]{L[s]}{}|&=|\mathcal{X}_{s}|\sum_{T\subset L}|\N[\mathcal{X}'_{L}]{S_T\cap L[s]}{}|\\
&=|\mathcal{X}_{s}|\sum_{T\subset L}\left(k-|T|\right)\\
&=|\mathcal{X}_{s}|\left(k(2^k-1)-(k\cdot 2^{k-1}-k)\right)\\
&=k\cdot 2^{k-1}|\mathcal{X}_{s}|.
\end{align*}

Now consider any $S \in \mathcal{X}_{s}$. If $k \geq 2$, then every leaf in $L$ is in $a_2(S)$ and $s \in N_2(S)$. So in this case $|a(S) \cap L[s]|=k$ and $|N(S) \cap L[s]|=1$. If $k=1$, then the one leaf in $L$ will either be in $a_1(S)$ or $a_2(S)$ depending on if $s$ has a non-leaf neighbour in $S$. Similarly, either $s \in N_1(S)$ or $s \in N_2(S)$ depending on if $s$ has a non-leaf neighbour in $S$. In both cases $|a(S) \cap L[s]|=k$ and $|N(S) \cap L[s]|=1$. Therefore, for every $k\geq 1$ and $S \in \mathcal{X}_{s}$ we have $|a(S) \cap L[s]|=k$ and $|N(S) \cap L[s]|=1$. Therefore 

$$|\A[\mathcal{X}_{s}]{L[s]}{}|=k|\mathcal{X}_{s}|  \text{ }\hspace{2mm}\text{ and }\hspace{2mm}\text{ } |\N[\mathcal{X}_{s}]{L[s]}{}|=|\mathcal{X}_{s}|.$$

Now combining the cases where $S \in \mathcal{X}_{s}$ or $S \in \mathcal{X}'_{L}$ we obtain

\vspace{-6mm}

\begin{align*}
|\A[\mathcal{X}'_{L} \cup \mathcal{X}_{s}]{L[s]}{}| &= |\A[\mathcal{X}'_{L}]{L[s]}{}|+|\A[\mathcal{X}_{s}]{L[s]}{}|=(2^{k}+k-1)|\mathcal{X}_{s}|\\
|\N[\mathcal{X}'_{L} \cup \mathcal{X}_{s}]{L[s]}{}| &= |\N[\mathcal{X}'_{L}]{L[s]}{}|+|\N[\mathcal{X}_{s}]{L[s]}{}|=(k \cdot 2^{k-1}+1)|\mathcal{X}_{s}|\\
\end{align*}

\vspace{-6mm}

\noindent Note that when $k=1,2$ we have $|\A[\mathcal{X}'_{L} \cup \mathcal{X}_{s}]{L[s]}{}|=|\N[\mathcal{X}'_{L} \cup \mathcal{X}_{s}]{L[s]}{}|$ and when $k \geq 3$ then $|\A[\mathcal{X}'_{L} \cup \mathcal{X}_{s}]{L[s]}{}|<|\N[\mathcal{X}'_{L} \cup \mathcal{X}_{s}]{L[s]}{}|$. Together with inequality $( \ref{eqn:leftovers})$ we obtain our desired result.

\end{proof}

In the next lemma we will strengthen the result from Lemma \ref{lem:a1n1} by showing the inequality $|a_1(S) \cap V'| \leq |N_1(S) \cap V_{I}|$ still holds when $V_{I} \subseteq V$ which excludes selected stems and their respective leaves.

\begin{lemma}
\label{lem:a1n1_v2}
Let $G$ be a graph with stems $s_1, s_2, \ldots, s_{\omega}$ and for some $I \subseteq \{s_1, s_2, \ldots, s_{\omega}\}$ let $V_{I} = V - \bigcup\limits_{s \in I}L[s]$. Then 

$$|\A[\mathcal{X}_I]{V_I}{1}| \leq |\N[\mathcal{X}_I]{V_I}{1}|.$$

%\noindent Additionally, if at least one $s \in I$ has a non-stem neighbour in $V_I$ then $|\A[\mathcal{X}_I]{V_I}{1}| < |\N[\mathcal{X}_I]{V_I}{1}|.$
\end{lemma}

\begin{proof}
It is sufficient to show that for every dominating set $S \in \mathcal{X}_I$ we have

$$|a_1(S) \cap V_{I}| \leq |N_1(S) \cap V_{I}|.$$

\noindent For a dominating set $S \in \mathcal{X}_I$ and any $u \in N_1(S)$, we have $N[u] \cap S \subseteq a_1(S)$ and $|N[u] \cap S|=1$. Let $f:N_1(S) \rightarrow a_1(S)$ be the mapping $f(u) = v$ where $v$ is the single element in $N[u] \cap S$. By the definition of $a_1(S)$, every $v \in a_1(S)$ has a neighbour in $N_1(S)$. Therefore $f:N_1(S) \rightarrow a_1(S)$ is a surjective mapping from $N_1(S)$ to $a_1(S)$. We will now restrict the domain of $f$. Note that it may be the case that the lone neighbour of a vertex $u \in N_1(S)$ may not be in $V_{I}$. Let $N_1'(S) \subseteq N_1(S)$ denote the subset of $N_1(S)$ which excludes all such vertices, that is,

$$N_1'(S)= \{u \in N_1(S) : f(u) \in V_{I}\}.$$

\noindent Restrict the domain of $f$ to be $N_1'(S)\cap V_{I}$. By the definition of $N_1'(S)$, for any $u \in N_1'(S)$ we have that $f(u) \in a_1(S)\cap V_{I}$. It is sufficient to show $f:N_1'(S)\cap V_{I} \rightarrow a_1(S)\cap V_{I}$ is a surjective mapping.

Let $v \in a_1(S)\cap V_{I}$. By the definition of $a_1(S)$, we have that $v$ has at least one neighbour $u \in N_1(S)$. To show a contradiction, suppose that $u \notin N_1'(S)\cap V_{I}$. Then it is either the case that $u \notin N_1'(S)$ or $u \notin V_{I}$. If $u \notin N_1'(S)$ then $f(u)=v \notin V_{I}$ which contradicts $v \in a_1(S)\cap V_{I}$. So suppose $u \notin V_{I}$. By the definition of $V_{I}$, we have that $u$ must then either be a leaf or a stem. If $u$ is a leaf then $v$, but $u$'s stem is not in $V_{I}$ which again contradicts $v \in a_1(S)\cap V_{I}$. So suppose $u$ is a stem which is not in $V_{I}$, i.e., $u\in I$.% \edit{ But since $S\in \mathcal{X}_I$, every stem in $I$ must be in $S$, a contradiction.(NOT NECESSARILY TRUE, SEE $\mathcal{X}_s$ FROM FIGURE 2.1(b) SO MUST KEEP FOLLOWING ARGUMENT)}
 As $u \in N_1(S)$ then $u \notin S$ and hence each leaf of $L(u)$ is in $S$, otherwise $S$ would not be a dominating set. However, this implies $u$ has neighbours $v$ and the vertices of $L(u)$ as neighbours in $S$.  This contracts the fact that $u \in N_1(S)$ and hence should have exactly one neighbour in $S$. Therefore $u \in V_{I}$ and hence $u \in N_1'(S)$. Thus for every $v \in a_1(S)\cap V_{I}$ there exists a $u \in N_1'(S)\cap V_I$ such that $f(u)=v$.

%We have now shown that $|\A[\mathcal{X}_I]{V_I}{1}| \leq |\N[\mathcal{X}_I]{V_I}{1}|$, we will now show the inequality is strict if at least one $s \in I$ has a non-stem neighbour in $V_I$. ***** Needs to be Non-domination covered vertex*****

\end{proof}

We call a graph \emph{star-like} if it can be formed by taking a smaller graph and adding one or two leaves to each of its vertices. Equivalently, a graph is star-like if every vertex is either a leaf or a $k$-stem where $k \leq 2$, see Figure~\ref{fig:star-like} for examples of star-like graphs.

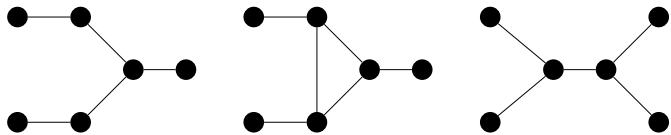
\begin{figure}
[h]
\def\c{0.7}
\def\h{2}
\centering
\scalebox{\c}{
\begin{tikzpicture}
\begin{scope}[every node/.style={circle,thick,draw,fill}]
    \node (1) at (1,0) {}; 
    \node (2) at (0,0) {}; 
    \node (3) at (-1,1) {}; 
    \node (4) at (-1,-1) {};
    \node (5) at (-2.2,1) {};  
    \node (6) at (-2.2,-1) {};
\end{scope}

\begin{scope}
    \path [-] (1) edge node {} (2);
	\path [-] (2) edge node {} (3);
	\path [-] (2) edge node {} (4);
	\path [-] (3) edge node {} (5);
	\path [-] (4) edge node {} (6);
\end{scope}
\end{tikzpicture}
\qquad
\begin{tikzpicture}
\begin{scope}[every node/.style={circle,thick,draw,fill}]
    \node (1) at (1,0) {}; 
    \node (2) at (0,0) {}; 
    \node (3) at (-1,1) {}; 
    \node (4) at (-1,-1) {};
    \node (5) at (-2.2,1) {};  
    \node (6) at (-2.2,-1) {};
\end{scope}

\begin{scope}
    \path [-] (1) edge node {} (2);
	\path [-] (2) edge node {} (3);
	\path [-] (2) edge node {} (4);
	\path [-] (3) edge node {} (4);
	\path [-] (3) edge node {} (5);
	\path [-] (4) edge node {} (6);
\end{scope}
\end{tikzpicture}
\qquad
\begin{tikzpicture}
\begin{scope}[every node/.style={circle,thick,draw,fill}]
    \node (1) at (1,1) {}; 
    \node (2) at (1,-1) {}; 
    \node (3) at (0,0) {}; 
    \node (4) at (-1,0) {}; 
    \node (5) at (-2.2,-1) {}; 
    \node (6) at (-2.2,1) {};
\end{scope}

\begin{scope}
    \path [-] (1) edge node {} (3);
	\path [-] (2) edge node {} (3);
	\path [-] (3) edge node {} (4);
	\path [-] (4) edge node {} (5);
	\path [-] (4) edge node {} (6);
\end{scope}
\end{tikzpicture}
}
\caption{All connected star-like graphs of order 6.}\label{fig:star-like}
\end{figure}

\begin{observation}[\cite{2021Beaton}]
\label{obs:StarLike}
For a star-like graph $G$ with $n$ vertices $ \avd(G) = \frac{2n}{3}$.
\end{observation}

 We are now ready to prove Conjecture \ref{conj:2n/3} and completely classify which graphs $G$ with no isolated vertices have $\avd(G) = \frac{2n}{3}$. 

\begin{theorem}
\label{thm:2n/3}
Let $G$ be a graph $n \geq 2$ vertices . If $G$ has no isolated vertices then 

$$\avd(G) \leq \frac{2n}{3}, $$

\noindent with equality if and only if $G$ is star-like. 
\end{theorem}

\begin{proof}
We begin by showing $\avd(G) \leq \frac{2n}{3}$. By Observation \ref{obs:AvN} and Observation \ref{obs:NewNotation} it suffices to show that $|\A{V}{}| \leq |\N{V}{}|$. Let $T_2 \subseteq V$ be the collection of non-stem vertices with degree at least two. By Lemma \ref{lem:deg2} we have  that for each $v \in T_2$

$$(2^{\deg(v)}-\deg(v)-1) \cdot |\A{v}{2}| \leq |\N{v}{2}|.$$

\noindent Therefore $|\A{T_2}{2}| \leq |\N{T_2}{2}|$. Furthermore by Lemma \ref{lem:a1n1_v2} we have

$$|\A{T_2}{1}| < |\N{T_2}{1}|.$$

\noindent Therefore $|\A{T_2}{}| < |\N{T_2}{}|$. If $G$ has minimum degree $\delta \geq 2$, then $V=T_2$ and hence $|\A{V}{}| < |\N{V}{}|$.

Suppose $G$ has minimum degree $\delta=1$ with stems $s_1, s_2, \ldots, s_{\omega}$. Note that we can partition the vertex set of $G$ as follows

$$V = T_2 \cup L[s_1] \cup L[s_2] \cup  \cdots \cup L[s_{\omega}].$$

\noindent Recall that for $I \subseteq \{s_1, s_2, \ldots, s_{\omega}\}$, $\mathcal{X}_{I}$ denotes the collection of dominating sets $S$ such that $L[s] \not\subseteq S$ for every $s \in I$ and $L[s] \subseteq S$ for every $s \notin I$. Recall that the collection of all $\mathcal{X}_{I}$ partitions $\mathcal{D}_G$.

From Lemma \ref{lem:kstem} we have that for any $I \subseteq \{s_1, s_2, \ldots, s_{\omega}\}$ and $k$-stem $s \in I$ that

$$|\A[\mathcal{X}_{I}]{L[s_i]}{}|\leq |\N[\mathcal{X}_{I}]{L[s_i]}{}|,$$

\noindent Additionally, from Lemma \ref{lem:a1n1_v2} for any $I \subseteq \{s_1, s_2, \ldots, s_{\omega}\}$ and every dominating set $S \in \mathcal{D}_G$ we have

 $$|\A[\mathcal{X}_I]{V_I}{1}| \leq |\N[\mathcal{X}_I]{V_I}{1}|.$$

\noindent where $V_{I} = V - \bigcup\limits_{s \in I}L[s]$. Therefore it suffices to show $|\A[\mathcal{X}_I]{V_I}{2}| \leq |\N[\mathcal{X}_I]{V_I}{2}|$ for any $I \subseteq \{s_1, s_2, \ldots, s_{\omega}\}$.  Moreover as $|\A{T_2}{2}| < |\N{T_2}{2}|$ it suffices to show that for any $S \in \mathcal{X}_I$ and $s \in I$ that $a_2(S) \cap L[s] = \emptyset$. Let $S \in \mathcal{X}_I$ and $s \in I$. By the definition of $\mathcal{X}_I$, we have that $L[s] \subseteq S$. Thus for each leaf $ \ell \in L(s)$, we have that $S- \ell$ is still a dominating set. Therefore $\ell \notin a(S)$ and hence $\ell \notin a_2(S)$. To show a contradiction suppose that $s \in a_2(S)$. Then, by the same argument used at the beginning of Lemma \ref{lem:deg2}, $S-s$ dominates every vertex in $G$ except $s$. However, $L[s] \subseteq S$ so $s$ would be dominated by one of its leaf neighbours which is a contradiction. Therefore $s \notin a_2(S)$ and $a_2(S) \cap L[s] = \emptyset$.

Finally we show that $\avd(G) = \frac{2n}{3}$ if and only if $G$ is star-like. By Observation \ref{obs:StarLike} it suffices to that if $\avd(G) = \frac{2n}{3}$ then $G$ is star-like. So suppose $\avd(G) = \frac{2n}{3}$. It follows from the same argument in Observation \ref{obs:AvN} that $|\A{V}{}| = |\N{V}{}|$. As $|\A{T_2}{2}| < |\N{T_2}{2}|$ then it must be the case that $T_2 = \emptyset$. Thus every vertex in $G$ must either be a leaf or a $k$-stem. By Lemma \ref{lem:a1n1_v2} it follows that for any $k$-stem $s$ with $k \geq 3$ that $|\A[\mathcal{X}_{I}]{L[s]}{}| < |\N[\mathcal{X}_{I}]{L[s]}{}|$.
Therefore as $|\A{V}{}| = |\N{V}{}|$ then $k \leq 2$ and $G$ is star-like.
\end{proof}

Clearly, every dominating set must contain every isolated vertex. Therefore we get the following corollary for graphs with $r$ isolated vertices.

\begin{corollary}
\label{cor:2n3boundwithiso}
Let $G$ be a graph $n$ vertices. If $G$ has $r$ isolated vertices then $\avd(G) \leq \frac{2n+r}{3}$.
\end{corollary}

\section{An average version of Vizing's Conjecture}\label{sec:Vizing}
Theorem \ref{thm:2n/3} implies a Vizing-like bound for the average order of dominating sets. The Vizing conjecture \cite{VizingConjecture} is one of the longest standing conjectures regarding the domination number and is among the most important conjectures in all of Graph Theory~\cite{BonatoNowakowski2012}. Vizing's conjecture posits that for two graphs $G$ and $H$ that

$$\gamma(G \Box H) \geq \gamma(G)\gamma(H),$$

\noindent where $G \Box H$ denotes the Cartesian product of $G$ and $H$. We will now should this bound holds for the average order of dominating sets. For this we will require a result from \cite{2021Beaton} that for two graphs $G$ and $H$ we have that $\avd(G \cup H) = \avd(G) + \avd(H)$.

\begin{theorem}
\label{thm:VizingAvd}
Let $G$ and $H$ be non-empty graphs, then $\avd(G \Box H) > \avd(G)\avd(H)$

\end{theorem}

\begin{proof}

Let $G'$ and $H'$ be isolate-free graphs such that $G \cong G' \cup r_1 K_1$ and  $H \cong H' \cup r_2 K_1$ where $r_1$ and $r_2$ denote the number of isolated vertices in $G$ and $H$ respectively. Moreover let $n_1$ and $n_2$ denote the number of vertices of $G'$ and $H'$ respectively. Note that

$$G \Box H \cong (G' \Box H') \cup r_1 H' \cup r_2 G' \cup r_1r_2K_1.$$

\noindent As $\avd(G)$ is additive across components then we have $\avd(G)=\avd(G')+r_1$ and $\avd(H)=\avd(H')+r_2$. Moreover

\begin{align*}
\avd(G \Box H) =& \avd(G' \Box H') + r_1 \avd(H')+ r_2 \avd(G')+ r_1r_2,\text{ and} \\ \\
\avd(G)\avd(H) = &  \avd(G')\avd(H') + r_1 \avd(H')+ r_2 \avd(G')+ r_1r_2. \\
\end{align*}

\noindent Therefore it suffices to show that $\avd(G' \Box H') \geq  \avd(G')\avd(H')$. Note that the $G' \Box H'$ has $n_1n_2$ vertices and therefore $\avd(G' \Box H') \geq \frac{n_1n_2}{2}$. Moreover by Theorem \ref{thm:2n/3} we have that $\avd(G') \geq \frac{2n_1}{3}$ and $\avd(H') \geq \frac{2n_2}{3}$. Finally, as $G$ and $H$ where not empty, then $n_1,n_2>0$ and

 $$\avd(G' \Box H') \geq \frac{n_1n_2}{2} > \frac{4n_1n_2}{9} \geq \avd(G')\avd(H').$$
 
\noindent Therefore $\avd(G \Box H) > \avd(G)\avd(H)$.
\end{proof}

%%%%%%%%%%%%%%%%%%%%%%%%%%%%%%%%%%%%%%%%%%%%%%%%%%%%%%%%%%%%%%%%%%%%%%%%%%%%%%%%%%%%%%%
\section{Conclusion and Open Problems}\label{sec:conclusion}
%%%%%%%%%%%%%%%%%%%%%%%%%%%%%%%%%%%%%%%%%%%%%%%%%%%%%%%%%%%%%%%%%%%%%%%%%%%%%%%%%%%%%%
This paper has been devoted to proving Conjecture~\ref{conj:2n/3}, and much of our work involved special handling of graphs with leaves. It is therefore reasonable to wonder what happens with $\avd(G)$ when $G$ does not have leaves, and more generally we bound the minimum degree.
\begin{problem}
Classify which graphs maximize $\avd(G)$ among all graphs of fixed order with $\delta\ge k$ for all $k\ge 2$.
\end{problem}
For graphs of small order, the graph with minimum degree at least two which maximizes $\avd(G)$ appears to be 2-regular. Particularly, for $3 \leq n \leq 7$, the graph $C_n$ maximizes $\avd(G)$ among graphs with $\delta(G) \geq 2$. However, for $n=8$ and $n=9$ it is $C_4 \cup C_4$ and $C_4 \cup C_5$ respectively which maximize $\avd(G)$. Moreover for graphs larger minimum degree, a computer search on $n \leq 9$ showed that the graphs which maximize $\avd(G)$ among those with minimum degree $\delta \geq k$ where $k$-regular (when possible). Surprisingly, connected graphs with $\delta \geq k$ do not follow this pattern. For $n=8$ the graph shown Figure \ref{fig:maxconnected} maximizes $\avd(G)$ among connected graphs with $\delta(G) \geq 2$.

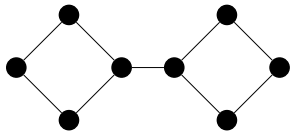
\begin{figure}
[h]
\def\c{0.7}
\def\h{2}
\centering
\scalebox{\c}{
\begin{tikzpicture}
\begin{scope}[every node/.style={circle,thick,draw,fill}]
    \node (1) at (0,0) {}; 
    \node (2) at (1,1) {}; 
    \node (3) at (2,0) {}; 
    \node (4) at (1,-1) {}; 
    \node (5) at (-1,0) {}; 
    \node (6) at (-2,1) {};
    \node (7) at (-3,0) {};
    \node (8) at (-2,-1) {};
    
\end{scope}

\begin{scope}
    \path [-] (1) edge node {} (2);
	\path [-] (2) edge node {} (3);
	\path [-] (3) edge node {} (4);
	\path [-] (4) edge node {} (1);
	\path [-] (1) edge node {} (5);
	\path [-] (5) edge node {} (6);
	\path [-] (6) edge node {} (7);
	\path [-] (7) edge node {} (8);
	\path [-] (8) edge node {} (5);
\end{scope}
\end{tikzpicture}
}
\caption{The graph which maximizes $\avd(G)$ among connected graphs of order 8 and $\delta(G) \geq 2$}\label{fig:maxconnected}
\end{figure}

%Beaton and Brown \cite{2021Beaton} showed that for a graph $G$ with order $n$ and minimum $\delta \geq 1$
%
%$$\avd(G) \leq \frac{n}{2} \left( 1 + \frac{\delta}{2^{\delta} -1} \right).$$
%
%\noindent For $\delta=1, 2 \text{ or } 3$ we have now improved this upperbound to $2n/3$.

A long-standing conjecture related to dominating sets is that the domination polynomial whose coefficients are the numbers of dominating sets of each cardinality is unimodal (that is, first non-decreasing and then non-increasing) for all graphs~\cite{2014Alikhani}.
We remark that when the domination polynomial has all real roots, the average order of domination sets of graph $G$ can also determine the mode of the coefficients of the domination polynomial $D(G,x)$; Darroch \cite{1964Darroch} showed in general that a real-rooted polynomial $f(x) = a_0 + a_1x + \cdots + a_nx^n$ with positive coefficients has its mode at either $\left\lfloor \frac{f'(1)}{f(1)} \right\rfloor$ or $\left\lceil \frac{f'(1)}{f(1)} \right\rceil$. Since $\avd(G) = \frac{D'(G,1)}{D(G,1)}$, it follows that if $D(G,x)$ has all real roots then its mode is at either $\left\lfloor \avd(G) \right\rfloor$ or $\left\lceil \avd(G) \right\rceil$. While the roots of domination polynomials are not always real and are in fact dense in the complex plane~\cite{2014BrownTufts}, Oboudi \cite{2015Oboudi} conjectured that star-like graphs were the only domination polynomials with all real roots. The connection between the mode of the domination polynomial of $G$ and $\avd(G)$ together with computational evidence leads us to pose the following conjecture.

\begin{conjecture}
For each order $n$, there is a star-like graph whose domination polynomial has its mode at the largest index among all graphs of order $n$ without isolated vertices.
\end{conjecture}

Lastly, Theorem  \ref{thm:VizingAvd} states that for two non-empty graphs $\avd(G \Box H) > \avd(G)\avd(H)$. This leads us to wonder if we can find upper and lower bounds for $\avd(G)$ in terms of $\gamma(G)$. Clearly $\avd(G) \geq \gamma(G)$ but this is only tight for $K_n$ and $\overline{K_n}$. Tighter bounds could be used to compare $\gamma(G \Box H)$ and  $\gamma(G)\gamma(H)$ and hence make progress on Vizing's Conjecture.

\begin{problem}
Find tight upper and lower bounds for $\avd(G)$ in terms of $\gamma(G)$.
\end{problem}

\section*{Acknowledgements}
 
\noindent Ben Cameron acknowledges research support from the Natural Sciences and Engineering Research Council of Canada (NSERC) (grants DGECR-2022-00446 and RGPIN-2022-03697)

%\section*{References}

\bibliographystyle{abbrv}
\bibliography{MDO}

\end{document}